\documentclass[reqno,a4paper,11pt]{amsart}
\usepackage[ngerman, english]{babel}
\usepackage[all,poly]{xy}
\usepackage[mathcal]{eucal}
\usepackage{enumerate}
\usepackage{amsmath,amssymb, amsthm, amsfonts}
\usepackage[pagebackref,colorlinks]{hyperref}
\usepackage[foot]{amsaddr}

\theoremstyle{plain}
\newtheorem {lemma}{Lemma}

\newtheorem {theorem}[lemma]{Theorem}
\newtheorem {corollary}[lemma]{Corollary}

\theoremstyle{definition}
\newtheorem{definition}[lemma]{Definition}

\newtheorem{remark}[lemma]{Remark}
\newtheorem {example}[lemma]{Example}

\title[Sandwich classification revisited]{Sandwich classification for $GL_n(R)$, $O_{2n}(R)$ and $U_{2n}(R,\Lambda)$ revisited}
\author{Raimund Preusser}
\address{Department of Mathematics,
University of Brasilia, Brazil}
\email{raimund.preusser@gmx.de}

\begin{document}
\maketitle
\begin{abstract}
Let $n$ be a natural number greater or equal to $3$, $R$ a commutative ring and $\sigma\in GL_n(R)$. We show that $t_{kl}(\sigma_{ij})$ (resp. $t_{kl}(\sigma_{ii}-\sigma_{jj})$) where $i\neq j$ and $k\neq l$ can be expressed as a product of $8$ (resp. $24$) matrices of the form $^{\epsilon}\sigma^{\pm 1}$ where $\epsilon\in E_n(R)$. We prove similar results for the orthogonal groups $O_{2n}(R)$ and the hyperbolic unitary groups $U_{2n}(R,\Lambda)$ under the assumption that $R$ is commutative and $n\geq 3$. This yields new, very short proofs of the Sandwich Classification Theorems for the groups $GL_n(R)$, $O_{2n}(R)$ and $U_{2n}(R,\Lambda)$.
\end{abstract}
\section{Introduction}
Let $n$ be a natural number greater or equal to $3$ and $R$ a commutative ring. Let $\sigma\in GL_n(R)$ and set $H:={}^{E_n(R)}\sigma$, i.e. $H$ is the smallest subgroup of $GL_n(R)$ which contains $\sigma$ and is normalized by $E_n(R)$. Let $I$ be the ideal of $R$ defined by $I:=\{x\in R\mid t_{12}(x)\in H\}$. Then clearly $E_n(R,I)\subseteq H$. By the Sandwich Classification Theorem (SCT) for $GL_n(R)$ one also has $H\subseteq C_n(R,I)$. It follows that $\sigma_{ij}, \sigma_{ii}-\sigma_{jj}\in I$ for any $i\neq j$, i.e. the matrices $t_{12}(\sigma_{ij})$ and $t_{12}(\sigma_{ii}-\sigma_{jj})$ can be expressed as products of matrices of the form $^{\epsilon}\sigma^{\pm 1}$ where $\epsilon\in E_n(R)$. We show how one can use the theme of the paper \cite{SV} in order to find such expressions and give boundaries for the number of factors, see Theorem \ref{mthm1}. This yields a new, very simple proof of the SCT for $GL_n(R)$.

Further we prove an orthogonal and a unitary version of Theorem \ref{mthm1} (cf. Theorem \ref{mthm2} and Theorem \ref{mthm3}). The proof of the orthogonal version is very simple. The proof of the unitary version is a bit more complicated, but still it is much shorter than the proof of the SCT for the groups $U_{2n}(R,\Lambda)$ given in \cite{preusser} (on the other hand, in \cite{preusser} the ring $R$ is only assumed to be quasi-finite and hence the result is a bit more general). For the hyperbolic unitary groups $U_{2n}(R,\Lambda)$ this yields the first proof of the SCT which does not use localization.

This paper is organized as follows. In Section \ref{sec2} we recall some standard notation
which will be used throughout the paper. In Section \ref{secpre} we state two lemmas which will be used in the proofs of the main theorems \ref{mthm1}, \ref{mthm2} and \ref{mthm3}. In Section \ref{sec3} we recall the definitions of the general linear group $GL_n(R)$ and some important subgroups, in Section \ref{sec4} we prove Theorem \ref{mthm1}. In Section \ref{sec5} we recall the definitions of the (even-dimensional) orthogonal group $O_{2n}(R)$ and some important subgroups, in Section \ref{sec6} we prove Theorem \ref{mthm2}. In Section \ref{sec7} we recall the definitions of A. Bak's hyperbolic unitary group $U_{2n}(R,\Lambda)$ and some important subgroups and in the last section we prove Theorem \ref{mthm3}.

\section{Notation}\label{sec2} 
By a natural number we mean an element of the set $\mathbb{N}:=\{1,2,3,\dots\}$. If $G$ is a group and $g,h\in G$, we let $^hg:=hgh^{-1}$ and $[g,h]:=ghg^{-1}h^{-1}$. By a ring we will always mean an associative ring with $1$ such that $1\neq 0$. Ideal will mean two-sided ideal. If $X$ is a subset of a ring $R$, then we denote by $I(X)$ the ideal of $R$ generated by $X$. If $X=\{x\}$, then we may write $I(x)$ instead of $I(X)$. The set of all invertible elements in a ring $R$ is denoted by $R^*$. If $m$ and $n$ are natural numbers and $R$ is a ring, then the set of all $m\times n$ matrices with entries in $R$ is denoted by $M_{m\times n}(R)$. If $a\in M_{m\times n}(R)$, we denote the transpose of $a$ by $a^t$ and the entry of $a$ at position $(i,j)$ by $a_{ij}$. We denote the $i$-th row of $a$ by $a_{i*}$ and its $j$-th column by $a_{*j}$. We set $M_n(R):=M_{n\times n}(R)$. The identity matrix in $M_n(R)$ is denoted by $e$ or $e^{n\times n}$ and the matrix with a $1$ at position $(i,j)$ and zeros elsewhere is denoted by $e^{ij}$. If $a\in M_n(R)$ is invertible, the entry 
of $a^{-1}$ at position $(i,j)$ is denoted by $a'_{ij}$, the $i$-th row of $a^{-1}$ by $a'_{i*}$ and the $j$-th column of $a^{-1}$ by $a'_{*j}$. Further we denote by $^nR$ the set of all rows $v=(v_1,\dots,v_n)$ with entries in $R$ and by $R^n$ the set of all columns $u=(u_1,\dots,u_n)^t$ with entries in $R$. We consider $^nR$ as left $R$-module and $R^n$ as right $R$-module.

\section{Preliminaries}\label{secpre}
The following two lemmas are easy to check.
\begin{lemma}\label{pre}
Let $G$ be a group and $a,b,c\in G$. Then ${}^{b^{-1}}[a,bc]=[b^{-1},a][a,c]$.
\end{lemma}
\begin{lemma}\label{pre2}
Let $G$ be a group, $E$ a subgroup and $a\in G$. Suppose that $b\in G$ is a product of $n$ elements of the form $^{\epsilon}a^{\pm 1}$ where $\epsilon\in E$. Then 
\begin{enumerate}[(i)]
\item $^{\epsilon'}b$ is a product of $n$ elements of the form $^{\epsilon}a^{\pm 1}$
\item $[\epsilon',b]$ is a product of $2n$ elements of the form $^{\epsilon}a^{\pm 1}$
\end{enumerate}
for any $\epsilon'\in E$.
\end{lemma}
Lemma \ref{pre2} will be used in the proofs of the main theorems without explicit reference.
\section{The general linear group $GL_n(R)$}\label{sec3}
In this section $n$ denotes a natural number, $R$ a ring and $I$ an ideal of $R$. We shall recall the definitions of the general linear group $GL_n(R)$ and the following subgroups of $GL_n(R)$; the elementary subgroup $E_n(R)$, the preelementary subgroup $E_n(I)$ of level $I$, the elementary subgroup $E_n(R,I)$ of level $I$, the principal congruence subgroup $GL_n(R,I)$ of level $I$ and the full congruence subgroup $C_n(R,I)$ of level $I$. 
\subsection{The general linear group}
\begin{definition}
$GL_{n}(R):=(M_n(R))^*$ is called {\it general linear group}.
\end{definition}
\subsection{The elementary subgroup}
\begin{definition}
Let $i,j\in\{1,\dots,n\}$ such that $i\neq j$ and $x\in R$. Then $t_{ij}(x):=e+xe^{ij}$ is called an {\it elementary transvection}. The subgroup of $GL_n(R)$ generated by all elementary transvections is called {\it elementary subgroup} and is denoted by $E_n(R)$. An elementary transvection $t_{ij}(x)$ is called $I$-{\it elementary} if $x\in I$. The subgroup of $GL_n(R)$ generated by all $I$-elementary transvections is called {\it preelementary subgroup of level} $I$ and is denoted by $E_n(I)$. Its normal closure in $E_n(R)$ is called {\it elementary subgroup of level} $I$ and is denoted by $E_n(R,I)$.
\end{definition}
\begin{lemma}\label{3}
The relations
\begin{align*}
t_{ij}(x)t_{ij}(y)&=t_{ij}(x+y), \tag{R1}\\
[t_{ij}(x),t_{hk}(y)]&=e \tag{R2}\text{ and}\\
[t_{ij}(x),t_{jk}(y)]&=t_{ik}(xy) \tag{R3}\\
\end{align*}
hold where $i\neq k, j\neq h$ in $(R2)$ and $i\neq k$ in $(R3)$.
\end{lemma}
\begin{proof}
Straightforward computation.
\end{proof}
\begin{definition}\label{4}
Let $i,j\in\{1,\dots,n\}$ such that $i\neq j$. Define $p_{ij}:=e+e^{ij}-e^{ji}-e^{ii}-e^{jj}=t_{ij}(1)t_{ji}(-1)t_{ij}(1)\in E_{n}(R)$. It is easy show that $p_{ij}^{-1}=p_{ji}$. 
\end{definition}
\begin{lemma}\label{5}
Let $x\in R$ and $i,j,k\in\{1,\dots,n\}$ be pairwise distinct indices. Then 
\begin{enumerate}[(i)]
\item $^{p_{ki}}t_{ij}(x)=t_{kj}(x)$ and
\item $^{p_{kj}}t_{ij}(x)=t_{ik}(x)$.
\end{enumerate}
\end{lemma}
\begin{proof}
Follows from the relations in Lemma \ref{3}.
\end{proof}
\subsection{Congruence subgroups}
\begin{definition}
The kernel of the group homomorphism $GL_n(R)\rightarrow GL_n(R/I)$ induced by the canonical map $R\rightarrow R/I$ is called {\it principal congruence subgroup of level} $I$ and is denoted by $GL_n(R,I)$. Obviously $GL_n(R,I)$ is a normal subgroup of $GL_n(R)$.
\end{definition}
\begin{definition}
The preimage of $Center(GL_n(R/I))$ under the group homomorphism $GL_n(R)\rightarrow GL_n(R/I)$ induced by the canonical map $R\rightarrow R/I$ is called {\it full congruence subgroup of level} $I$ and is denoted by $C_n(R,I)$. Obviously $GL_n(R,I)\subseteq C_n(R,I)$ and $C_n(R,I)$ is a normal subgroup of $GL_n(R)$.
\end{definition}
\begin{theorem}\label{8}
If $n\geq 3$ and $R$ is almost commutative (i.e. module finite over its center), then the equalities
\begin{align*}
[C_n(R,I),E_n(R)]=[E_n(R,I),E_n(R)]=E_n(R,I)
\end{align*}
hold.
\end{theorem}
\begin{proof}
See \cite{vaserstein}, Corollary 14.
\end{proof}
\section{Sandwich classification for $GL_n(R)$}\label{sec4}
In this section $n$ denotes a natural number greater or equal to $3$ and $R$ a commutative ring.
\begin{definition}
Let $\sigma\in GL_n(R)$. Then a matrix of the form $^{\epsilon}\sigma^{\pm 1}$ where $\epsilon\in E_n(R)$ is called an {\it elementary $\sigma$-conjugate}.
\end{definition}
\begin{theorem}\label{mthm1}
Let $\sigma\in GL_n(R)$, $i\neq j$ and $k\neq l$. Then 
\begin{enumerate}[(i)]
\item $t_{kl}(\sigma_{ij})$ is a product of $8$ elementary $\sigma$-conjugates and
\item $t_{kl}(\sigma_{ii}-\sigma_{jj})$ is a product of $24$ elementary $\sigma$-conjugates.
\end{enumerate}
\end{theorem}
\begin{proof}
(i) Set $\tau:=t_{21}(-\sigma_{23})t_{31}(\sigma_{22})$. One checks easily that the second row of $\sigma\tau^{-1}$ equals the second row of $\sigma$ and hence the second row of 
$\xi:={}^{\sigma}\tau^{-1}$ is trivial. Set  
\begin{align*}
\zeta:={}^{\tau^{-1}}[t_{32}(1),[\tau,\sigma]]={}^{\tau^{-1}}[t_{32}(1),\tau\xi]\overset{L.\ref{pre}}{=}[\tau^{-1},t_{32}(1)][t_{32}(1),\xi].
\end{align*}
One checks easily that $[\tau^{-1},t_{32}(1)]=t_{31}(-\sigma_{23})$ and $[t_{32}(1),\xi]=\prod\limits_{i\neq 2 }t_{i2}(x_i)$ for some $x_1,x_3,x_4,\dots,x_n\in R$. Hence $\zeta=t_{31}(-\sigma_{23})\prod\limits_{i\neq 2 }t_{i2}(x_i)$. It follows that $[t_{12}(1),\zeta]=t_{32}(\sigma_{23})$. Hence we have shown
\[[t_{12}(1),{}^{t_{21}(\sigma_{23})t_{31}(-\sigma_{22})}[t_{32}(1),[t_{21}(-\sigma_{23})t_{31}(\sigma_{22}),\sigma]]]=t_{32}(\sigma_{23}).\]
This implies that $t_{32}(\sigma_{23})$ is a product of $8$ elementary $\sigma$-conjugates. It follows from Lemma \ref{5} that $t_{kl}(\sigma_{23})$ is a product of $8$ elementary $\sigma$-conjugates. Since one can bring $\sigma_{ij}$ to position $(2,3)$ by conjugating monomial matrices in $E_n(R)$ (see Definition \ref{4}) to $\sigma$, the assertion of (i) follows.\\
\\
(ii) Clearly the entry of $^{t_{ji}(1)}\sigma$ at position $(j,i)$ equals $\sigma_{ii}-\sigma_{jj}+\sigma_{ji}-\sigma_{ij}$. Applying (i) to $^{t_{ji}(1)}\sigma$ we get that $t_{kl}(\sigma_{ii}-\sigma_{jj}+\sigma_{ji}-\sigma_{ij})$ is a product of $8$ elementary $\sigma$-conjugates (note that any elementary $^{t_{ji}(1)}\sigma$-conjugate is also an elementary $\sigma$-conjugate). Applying (i) to $\sigma$ we get that $t_{kl}(\sigma_{ij}-\sigma_{ji})=t_{kl}(\sigma_{ij})t_{kl}(-\sigma_{ji})$ is a product of $16$ elementary $\sigma$-conjugates. It follows that $t_{kl}(\sigma_{ii}-\sigma_{jj})=t_{kl}(\sigma_{ii}-\sigma_{jj}+\sigma_{ji}-\sigma_{ij})t_{kl}(\sigma_{ij}-\sigma_{ji})$ is a product of $24$ elementary $\sigma$-conjugates.
\end{proof}

As a corollary we get the Sandwich Classification Theorem for $GL_n(R)$. Note that if $\sigma\in GL_n(R)$ and $I$ is an ideal of $R$, then $\sigma\in C_n(R,I)$ if and only if $\sigma_{ij},\sigma_{ii}-\sigma_{jj}\in I$ for any $i\neq j$.
\begin{corollary}
Let $H$ be a subgroup of $GL_n(R)$. Then $H$ is normalized by $E_n(R)$ if and only if 
\begin{equation}
E_n(R,I)\subseteq H\subseteq C_n(R,I)
\end{equation}
for some ideal $I$ of $R$.
\end{corollary}
\begin{proof}
First suppose that $H$ is normalized by $E_n(R)$. Let $I$ be the ideal of $R$ defined by $I:=\{x\in R\mid t_{12}(x)\in H\}$. Then clearly $E_n(R,I)\subseteq H$. It remains to show that $H\subseteq C_n(R,I)$, i.e. that if $\sigma\in H$, then $\sigma_{ij},\sigma_{ii}-\sigma_{jj}\in I$ for any $i\neq j$. But that follows from the previous theorem. Suppose now that (1) holds for some ideal $I$. Then it follows from the standard commutator formulas in Theorem \ref{8} that $H$ is normalized by $E_n(R)$.
\end{proof}

\section{The even-dimensional orthogonal group $O_{2n}(R)$}\label{sec5}
In this section $n$ denotes a natural number, $R$ a commutative ring and $I$ an ideal of $R$. We shall recall the definitions of the even-dimensional orthogonal group $O_{2n}(R)$ and the following subgroups of $O_{2n}(R)$; the elementary subgroup $EO_{2n}(R)$, the preelementary subgroup $EO_{2n}(I)$ of level $I$, the elementary subgroup $EO_{2n}(R,I)$ of level $I$, the principal congruence subgroup $O_{2n}(R,I)$ of level $I$, and the full congruence subgroup $CO_{2n}(R,I)$ of level $I$.
\subsection{The even-dimensional orthogonal group}
\begin{definition}
Set $V:=R^{2n}$. We use the following indexing for the elements of the standard basis of $V$: $(e_1,\dots,e_n,e_{-n},\dots,e_{-1})$.
That means that $e_i$ is the column whose $i$-th coordinate is one and all the other coordinates are zero if $1\leq i\leq n$ and the column whose $(2n+1+i)$-th coordinate is one and all the other coordinates are zero if $-n\leq i \leq -1$. Let $p\in M_n(R)$ be the matrix with ones on the skew diagonal and zeros elsewhere. We define the quadratic form 
\begin{align*}
q:V&\rightarrow R\\v&\mapsto v^t\begin{pmatrix} 0 & p \\0& 0 \end{pmatrix}v.
\end{align*} 
The subgroup $O_{2n}(R):=\{\sigma\in GL_{2n}(R)\mid q(\sigma v)=q(v) ~\forall v\in V\}$ of $GL_{2n}(R)$ is called {\it (even-dimensional) orthogonal group}.
\end{definition}
\begin{remark}
The even-dimensional orthogonal groups are special cases of the hyperbolic unitary groups, cf. Example \ref{30}.
\end{remark}
\begin{definition}
We define $\Omega:=\{1,...,n,-n,...,-1\}$.
\end{definition}
\begin{lemma}
Let $\sigma\in GL_{2n}(R)$. Then $\sigma\in O_{2n}(R)$ if and only if
\begin{enumerate}[(i)]
\item $\sigma'_{ij}=\sigma_{-j,-i}~ \forall i,j\in\Omega$  and
\item $q(\sigma_{*j})=0 ~\forall j\in\Omega$. 
\end{enumerate}
\end{lemma}
\begin{proof}
See \cite{bak-vavilov}, p.167.
\end{proof}
\begin{lemma}\label{16}
Let $\sigma\in O_{2n}(R)$, $x \in R^*$ and $k\in \Omega$. Then the statements below are true.
\begin{enumerate}[(i)]
\item If the $k$-th column of $\sigma$ equals $e_kx$ then the $(-k)$-th row of $\sigma$ equals $x^{-1}e^t_{-k}$.
\item If the $k$-th row of $\sigma$ equals $xe^t_k$ then the $(-k)$-th column of $\sigma$ equals $e_{-k}x^{-1}$.
\end{enumerate}
\end{lemma}
\begin{proof}
Follows from (i) in the previous lemma.
\end{proof}
\subsection{The elementary subgroup}
\begin{definition}
If $i,j\in\Omega$ such that $i\neq \pm j$ and $x\in R$, then the matrix
\[T_{ij}(x):=e+x e^{ij}-x e^{-j,-i}\in O_{2n}(R)\]
is called an {\it elementary orthogonal transvection}. The subgroup of $O_{2n}(R)$ generated by all elementary orthogonal transvections is called {\it elementary orthogonal group} and is denoted by $EO_{2n}(R)$. An elementary orthogonal transvection $T_{ij}(x)$ is called $I$-{\it elementary} if $x\in I$. The subgroup of $O_{2n}(R)$ generated by all $I$-elementary orthogonal transvections is called {\it preelementary  subgroup of level $I$} and is denoted by $EO_{2n}(I)$. Its normal closure in $EO_{2n}(R)$ is called {\it elementary subgroup of level $I$} and is denoted by $EO_{2n}(R,I)$. 
\end{definition}
\begin{lemma}\label{18}
The relations
\begin{align*}T_{ij}(x)&=T_{-j,-i}(-x)\tag{R1},\\
T_{ij}(x)T_{ij}(y)&=T_{ij}(x+y)\tag{R2},\\
[T_{ij}(x),T_{hk}(y)]&=e,\tag{R3}\\
[T_{ij}(x),T_{jk}(y)]&=T_{ik}(xy),\tag{R4}\\
[T_{ij}(x),T_{j,-i}(y)]&=e\tag{R5}\end{align*}
hold where $h\neq j,-i$ and $k\neq i,-j$ in \textnormal{(R3)} and $i\neq \pm k$ in \textnormal{(R4)}.
\end{lemma}
\begin{proof}
Straightforward calculation.
\end{proof}
\begin{definition}\label{19}
Let $i,j\in\Omega$ such that $i\neq\pm j$. Define $P_{ij}:=e+e^{ij}-e^{ji}+e^{-i,-j}-e^{-j,-i}-e^{ii}-e^{jj}-e^{-i,-i}-e^{-j,-j}=T_{ij}(1)T_{ji}(-1)T_{ij}(1)\in EO_{2n}(R)$. It is easy show that $(P_{ij})^{-1}=P_{ji}$.
\end{definition}
\begin{lemma}\label{20}
Let $x\in R$ and $i,j,k\in\Omega$ such that $i\neq \pm j$ and $k\neq \pm i,\pm j$. Then 
\begin{enumerate}[(i)]
\item $^{P_{ki}}T_{ij}(x)=T_{kj}(x)$ and
\item $^{P_{kj}}T_{ij}(x)=T_{ik}(x)$.
\end{enumerate}
\end{lemma}
\begin{proof}
Follows from the relations in Lemma \ref{18}.
\end{proof}
\subsection{Congruence subgroups}
\begin{definition}
The kernel of the group homomorphism $O_{2n}(R)\rightarrow O_{2n}(R/I)$ induced by the canonical map $R\rightarrow R/I$ is called {\it principal congruence subgroup of level} $I$ and is denoted by $O_{2n}(R,I)$. Obviously $O_{2n}(R,I)$ is a normal subgroup of $O_{2n}(R)$.
\end{definition}
\begin{definition}
The preimage of $Center(O_{2n}(R/I))$ under the group homomorphism $O_{2n}(R)\rightarrow O_{2n}(R/I)$ induced by the canonical map $R\rightarrow R/I$ is called {\it full congruence subgroup of level} $I$ and is denoted by $CO_{2n}(R,I)$. Obviously $O_{2n}(R,I)\subseteq CO_{2n}(R,I)$ and $CO_{2n}(R,I)$ is a normal subgroup of $O_{2n}(R)$.
\end{definition}
\begin{theorem}\label{23}
If $n\geq 3$, then the equalities
\begin{align*}[CO_{2n}(R,I),EO_{2n}(R)]=[EO_{2n}(R,I),EO_{2n}(R)]=EO_{2n}(R,I)\end{align*}
hold.
\end{theorem}
\begin{proof}
See \cite{bak-vavilov}, Theorem 1.1 and Lemma 5.2.
\end{proof}
\section{Sandwich classification for $O_{2n}(R)$}\label{sec6}
In this section $n$ denotes a natural number greater or equal to $3$ and $R$ a commutative ring.
\begin{definition}
Let $\sigma\in O_{2n}(R)$. Then a matrix of the form $^{\epsilon}\sigma^{\pm 1}$ where $\epsilon\in EO_{2n}(R)$ is called an {\it elementary (orthogonal) $\sigma$-conjugate}.
\end{definition}
\begin{theorem}\label{mthm2}
Let $\sigma\in O_{2n}(R)$, $i\neq \pm j$ and $k\neq \pm l$. Then 
\begin{enumerate}[(i)]
\item $T_{kl}(\sigma_{ij})$ is a product of $8$ elementary orthogonal $\sigma$-conjugates,
\item $T_{kl}(\sigma_{i,-i})$ is a product of $16$ elementary orthogonal $\sigma$-conjugates,
\item $T_{kl}(\sigma_{ii}-\sigma_{jj})$ is a product of $24$ elementary orthogonal $\sigma$-conjugates and
\item $T_{kl}(\sigma_{ii}-\sigma_{-i,-i})$ is a product of $48$ elementary orthogonal $\sigma$-conjugates.
\end{enumerate}
\end{theorem}
\begin{proof}
(i) Set $\tau:=T_{21}(-\sigma_{23})T_{31}(\sigma_{22})T_{2,-3}(\sigma_{2,-1})$. One checks easily that the second row of $\sigma\tau^{-1}$ equals the second row of $\sigma$ and hence the second row of 
$\xi:={}^{\sigma}\tau^{-1}$ is trivial. By Lemma \ref{16} the second last column of $\xi$ also is trivial. Set  
\begin{align*}
\zeta:={}^{\tau^{-1}}[T_{32}(1),[\tau,\sigma]]={}^{\tau^{-1}}[T_{32}(1),\tau\xi]\overset{L.\ref{pre}}{=}[\tau^{-1},T_{32}(1)][T_{32}(1),\xi].
\end{align*}
One checks easily that $[\tau^{-1},T_{32}(1)]=T_{31}(-\sigma_{23})$ and
$[T_{32}(1),\xi]=\prod\limits_{i\neq \pm 2 }T_{i2}(x_i)$ for some $x_i\in R~(i\neq \pm 2)$. Hence $\zeta=T_{31}(-\sigma_{23})\prod\limits_{i\neq \pm 2 }T_{i2}(x_i)$. It follows that $[T_{12}(1),\zeta]=T_{32}(\sigma_{23})$. Hence we have shown
\[[T_{12}(1),{}^{T_{21}(\sigma_{23})T_{31}(-\sigma_{22})T_{2,-3}(-\sigma_{2,-1})}[T_{32}(1),[T_{21}(-\sigma_{23})T_{31}(\sigma_{22})T_{2,-3}(\sigma_{2,-1}),\sigma]]]=T_{32}(\sigma_{23}).\]
This implies that $T_{32}(\sigma_{23})$ is a product of $8$ elementary $\sigma$-conjugates. It follows from Lemma \ref{20} that $T_{kl}(\sigma_{23})$ is a product of $8$ elementary $\sigma$-conjugates. Since one can bring $\sigma_{ij}$ to position $(2,3)$ by conjugating monomial matrices in $EO_{2n}(R)$ (see Definition \ref{19}) to $\sigma$, the assertion of (i) follows.\\
\\
(ii) Clearly the entry of $^{T_{ji}(1)}\sigma$ at position $(j,-i)$ equals $\sigma_{i,-i}+\sigma_{j,-i}$. Applying (i) to $^{T_{ji}(1)}\sigma$ we get that $T_{kl}(\sigma_{i,-i}+\sigma_{j,-i})$ is a product of $8$ elementary $\sigma$-conjugates (note that any elementary $^{T_{ji}(1)}\sigma$-conjugate is also an elementary $\sigma$-conjugate). Applying (i) to $\sigma$ we get that $T_{kl}(\sigma_{j,-i})$ is a product of $8$ elementary $\sigma$-conjugates. It follows that $T_{kl}(\sigma_{i,-i})=T_{kl}(\sigma_{i,-i}+\sigma_{j,-i})T_{ji}(-\sigma_{j,-i})$ is a product of $16$ elementary $\sigma$-conjugates.\\
\\
(iii) Clearly the entry of $^{T_{ji}(1)}\sigma$ at position $(j,i)$ equals $\sigma_{ii}-\sigma_{jj}+\sigma_{ji}-\sigma_{ij}$. Applying (i) to $^{T_{ji}(1)}\sigma$ we get that $T_{kl}(\sigma_{ii}-\sigma_{jj}+\sigma_{ji}-\sigma_{ij})$ is a product of $8$ elementary $\sigma$-conjugates (note that any elementary $^{T_{ji}(1)}\sigma$-conjugate is also an elementary $\sigma$-conjugate). Applying (i) to $\sigma$ we get that $T_{kl}(\sigma_{ij}-\sigma_{ji})=T_{kl}(\sigma_{ij})T_{kl}(-\sigma_{ji})$ is a product of $16$ elementary $\sigma$-conjugates. It follows that $T_{kl}(\sigma_{ii}-\sigma_{jj})=T_{kl}(\sigma_{ii}-\sigma_{jj}+\sigma_{ji}-\sigma_{ij})T_{kl}(\sigma_{ij}-\sigma_{ji})$ is a product of $24$ elementary $\sigma$-conjugates.\\
\\
(iv) Follows from (iii) since $T_{kl}(\sigma_{ii}-\sigma_{-i,-i})=T_{kl}(\sigma_{ii}-\sigma_{jj})T_{kl}(\sigma_{jj}-\sigma_{-i,-i})$.\\
\\
\end{proof}

As a corollary we get the Sandwich Classification Theorem for $O_{2n}(R)$. Note that if $\sigma\in O_{2n}(R)$ and $I$ is an ideal of $R$, then $\sigma\in CO_{2n}(R,I)$ if and only if $\sigma_{ij},\sigma_{ii}-\sigma_{jj}\in I$ for any $i\neq j$.
\begin{corollary}
Let $H$ be a subgroup of $O_{2n}(R)$. Then $H$ is normalized by $EO_{2n}(R)$ if and only if 
\begin{equation}
EO_{2n}(R,I)\subseteq H\subseteq CO_{2n}(R,I)
\end{equation}
for some ideal $I$ of $R$.
\end{corollary}
\begin{proof}
First suppose that $H$ is normalized by $EO_{2n}(R)$. Let $I$ be the ideal of $R$ defined by $I:=\{x\in R\mid T_{12}(x)\in H\}$. Then clearly $EO_{2n}(R,I)\subseteq H$. It remains to show that $H\subseteq CO_{2n}(R,I)$, i.e. that if $\sigma\in H$, then $\sigma_{ij},\sigma_{ii}-\sigma_{jj}\in I$ for any $i\neq j$. But that follows from the previous theorem. Suppose now that (2) holds for some ideal $I$. Then it follows from the standard commutator formulas in Theorem \ref{23} that $H$ is normalized by $EO_{2n}(R)$
\end{proof}

\section{Bak's unitary group $U_{2n}(R,\Lambda)$}\label{sec7}
In order to classify the subgroups of a general linear group (resp. an even-dimensional orthogonal group) which are normalized by the elementary subgroup (resp. the elementary orthogonal group), the notion of an ideal of a ring is sufficient. Bak's dissertation \cite{bak_2} showed that the notion of an ideal by itself was not sufficient to solve the analogous classification problem for unitary groups, but that a refinement of the notion of an ideal, called a form ideal, was necessary. This led naturally to a more general notion of unitary group, which was defined over a form ring instead of just a ring and generalized all previous concepts. We describe form rings $(R,\Lambda)$ and form ideals $(I, \Gamma)$ first, then the hyperbolic unitary group $U_{2n}(R,\Lambda)$ and its elementary subgroup $EU_{2n}(R,\Lambda)$ over a form ring $(R,\Lambda)$. For a form ideal $(I,\Gamma)$, we recall the definitions of the following subgroups of $U_{2n}(R,\Lambda)$; the preelementary subgroup $EU_{2n}(I, \Gamma)$ of level $(I,\Gamma)$, the elementary subgroup $EU_{2n}((R,\Lambda),(I,\Gamma))$ of level $(I,\Gamma)$, the principal congruence subgroup $U_{2n}((R,\Lambda),(I,\Gamma))$ of level $(I,\Gamma)$, and the full congruence subgroup $CU_{2n}((R,\Lambda),(I,\Gamma))$ of level $(I,\Gamma)$.
\subsection{Form rings and form ideals}
\begin{definition}
Let $R$ be a ring and \begin{align*}\bar{}:R&\rightarrow R\\x&\mapsto \bar{x}\end{align*}
an involution on $R$, i.e. $\overline{x+y}=\bar{x}+\bar{y}$, $\overline{xy}=\bar{y}\bar{x}$ and $\bar{\bar{x}}=x$ for any $x,y\in R$. Let $\lambda\in center(R)$ such that $\lambda\bar{\lambda}=1$ and set $\Lambda_{min}:=\{x-\lambda\bar{x}\mid x\in R\}$ and $\Lambda_{max}:=\{x\in R\mid x=-\lambda\bar{x}\}$. An additive subgroup $\Lambda$ of $R$ such that 
\begin{enumerate}[(i)]
\item $\Lambda_{min}\subseteq\Lambda\subseteq\Lambda_{max}$ and
\item $x\Lambda\bar{x}\subseteq\Lambda~\forall x\in R$ 
\end{enumerate}
is called a {\it form parameter for $R$}. If $\Lambda$ is a form parameter for $R$, the pair $(R,\Lambda)$ is called a {\it form ring}.
\end{definition}
\begin{definition}
Let $(R,\Lambda)$ be a form ring and $I$ an ideal such that $\bar{I}=I$. Set $\Gamma_{max}=I\cap\Lambda$ and $\Gamma_{min}=\{x-\lambda\bar{x}\mid x\in I\}+\langle \{xy\bar{x}\mid x\in I,y\in\Lambda\}\rangle$. An additive subgroup $\Gamma$ of $I$ such that 
\begin{enumerate}[(i)]
\item $\Gamma_{min}\subseteq\Gamma\subseteq\Gamma_{max}$ and  
\item $x\Gamma\bar{x}\subseteq\Gamma~\forall x\in R$ 
\end{enumerate}
is called a {\it relative form parameter of level} $I$. If $\Gamma$ is a relative form parameter of level $I$, then $(I,\Gamma)$ is called a {\it form ideal of $(R,\Lambda)$}.
\end{definition}

Until the end of section 8 let $n\in\mathbb{N}$, $(R,\Lambda)$ a form ring and $(I,\Gamma)$ a form ideal of $(R,\Lambda)$.
\subsection{The hyperbolic unitary group}
\begin{definition}
Set $V:=R^{2n}$. We use the following indexing for the elements of the standard basis of $V$: $(e_1,\dots,e_n,e_{-n},\dots,e_{-1})$. That means that $e_i$ is the column whose $i$-th coordinate is one and all the other coordinates are zero if $1\leq i\leq n$ and the column whose $(2n+1+i)$-th coordinate is one and all the other coordinates are zero if $-n\leq i \leq -1$. Let $p\in M_n(R)$ be the matrix with ones on the skew diagonal and zeros elsewhere. We define the maps 
\begin{align*}
f:V\times V&\rightarrow R&h:V\times V&\rightarrow R&q:V&\rightarrow R/\Lambda\\(v,w)&\mapsto \overline{v}^t\begin{pmatrix} 0 & p \\ 0 & 0 \end{pmatrix}w,&(v,w)&\mapsto \overline{v}^t\begin{pmatrix} 0 & p \\ \lambda p & 0 \end{pmatrix}w,&v&\mapsto f(v,v)+\Lambda
\end{align*} 
where $\bar v$ is obtained from $v$ by applying $~\bar{}~$ to each entry of $v$.
For any $v\in V$, $f(v,v)$ is called the {\it value of $v$} and is denoted by $|v|$. The subgroup $U_{2n}(R,\Lambda):=\{\sigma\in GL_{2n}(R)\mid (h(\sigma u,\sigma v)=h(u,v) \wedge q(\sigma v)=q(v))~\forall u,v\in V\}$ of $GL_{2n}(R)$ is called {\it hyperbolic unitary group}.  
\end{definition}
\begin{example}\label{30}
If $R$ is commutative, $~\bar{}=id$, $\lambda=-1$ and $\Lambda=\Lambda_{max}=R$, then $U_{2n}(R,\Lambda)$ equals the symplectic group $Sp_{2n}(R)$. If $R$ is commutative, $~\bar{}~=id$, $\lambda=1$ and $\Lambda=\Lambda_{min}=\{0\}$, then $U_{2n}(R,\Lambda)$ equals the orthogonal group $O_{2n}(R)$.
\end{example}
\begin{definition}
We define $\Omega_+:=\{1,...,n\}$, $\Omega_-:=\{-n,...,-1\}$, $\Omega:=\Omega_+\cup\Omega_-$ and 
\begin{align*}
\epsilon:\Omega&\rightarrow\{-1,1\}\\
i&\mapsto\epsilon(i):=\begin{cases}
1, &if~i\in\Omega_+,\\
-1, &if~i\in\Omega_-. 
\end{cases}
\end{align*}
Further if $i,j\in \Omega$, we write $i<j$ iff either $i,j\in \Omega_+\land i<j$ or $i,j\in \Omega_-\land i<j$ or $i\in \Omega_+\land j\in \Omega_-$.
\end{definition}
\begin{lemma}\label{32}
Let $\sigma\in GL_{2n}(R)$. Then $\sigma\in U_{2n}(R,\Lambda)$ if and only if
\begin{enumerate}[(i)]
\item $\sigma'_{ij}=\lambda^{(\epsilon(j)-\epsilon(i))/2}\bar{\sigma}_{-j,-i}~ \forall i,j\in\Omega$  and
\item $|\sigma_{*j}|\in\Lambda ~\forall j\in\Omega$. 
\end{enumerate}
\end{lemma}
\begin{proof}
See \cite{bak-vavilov}, p.167.
\end{proof}
\begin{lemma}\label{33}
Let $\sigma\in U_{2n}(R,\Lambda)$, $x \in R^*$ and $k\in \Omega$. Then the statements below are true.
\begin{enumerate}[(i)]
\item If the $k$-th column of $\sigma$ equals $e_kx$ then the $(-k)$-th row of $\sigma$ equals $\overline{x^{-1}}e^t_{-k}$.
\item If the $k$-th row of $\sigma$ equals $x e^t_k$ then the $(-k)$-th column of $\sigma$ equals $e_{-k}\overline{x^{-1}}$.
\end{enumerate}
\end{lemma}
\begin{proof}
Follows from (i) in the previous lemma.
\end{proof}
\subsection{Polarity map}
\begin{definition}
The map
\begin{align*}
\widetilde{}: V&\longrightarrow{}^{2n}R\\
v&\longmapsto \begin{pmatrix}\lambda \bar v_{-1}&\dots&\lambda\bar v_{-n}&\bar v_{n}&\dots&\bar v_1\end{pmatrix}
\end{align*}
is called {\it polarity map}. One checks easily that $h(u,v)=\tilde uv$ for any $u,v\in V$ and that $~\widetilde{}~$ is {\it involutary linear}, i.e. $\widetilde{u+v}=\tilde u+\tilde v$ and $\widetilde{vx}=\bar x\tilde v$ for any $u,v\in V$ and $x\in R$.
\end{definition}
\begin{lemma}\label{35}
If $\sigma\in U_{2n}(R,\Lambda)$ and $v\in V$, then $\widetilde{\sigma v}=\tilde v\sigma^{-1}$.
\end{lemma}
\begin{proof}
See \cite[Lemma 2.5]{bak-vavilov}.
\end{proof}
\subsection{The elementary subgroup}
\begin{definition}
If $i,j\in\Omega$ such that $i\neq \pm j$ and $x\in R$, then the matrix
\[T_{ij}(x):=e+x e^{ij}-\lambda^{(\epsilon(j)-\epsilon(i))/2}\overline x e^{-j,-i}\in U_{2n}(R,\Lambda)\]
is called an {\it elementary short root transvection}. If $i\in \Omega$ and $y\in\lambda^{-(\epsilon(i)+1)/2}\Lambda$, then the matrix \[T_{i,-i}(y):=e+y e^{i,-i}\in U_{2n}(R,\Lambda)\] is called an {\it elementary long root transvection}. If $\sigma\in U_{2n}(R,\Lambda)$ is an elementary short root transvection or an elementary long root transvection, it is called an {\it elementary unitary transvection}. The subgroup of $U_{2n}(R,\Lambda)$ generated by all elementary unitary transvections is called {\it elementary unitary group} and is denoted by $EU_{2n}(R,\Lambda)$. An elementary unitary transvection $T_{ij}(x)$ is called {\it$(I,\Gamma)$-elementary} if $i\neq -j~\wedge~x\in I$ or $i=-j \wedge x\in\lambda^{-(\epsilon(i)+1)/2}\Gamma$. The subgroup of 
$U_{2n}(R,\Lambda)$ generated by all $(I,\Gamma)$-elementary transvections is called {\it preelementary subgroup of level $(I,\Gamma)$} and is denoted by $EU_{2n}(I,\Gamma)$. Its normal closure in $EU_{2n}(R,\Lambda)$ is called {\it elementary subgroup of level $(I,\Gamma)$} and is denoted by $EU_{2n}((R,\Lambda),(I,\Gamma))$. 
\end{definition}
\begin{lemma}\label{37}
The relations
\begin{align*}T_{ij}(x)&=T_{-j,-i}(-\lambda^{(\epsilon(j)-\epsilon(i))/2}\overline{x})\tag{R1},\\
T_{ij}(x)T_{ij}(y)&=T_{ij}(x+y)\tag{R2},\\
[T_{ij}(x),T_{hk}(y)]&=e,\tag{R3}\\
[T_{ij}(x),T_{jk}(y)]&=T_{ik}(xy),\tag{R4}\\
[T_{ij}(x),T_{j,-i}(y)]&=T_{i,-i}(xy-\lambda^{-\epsilon(i)}\bar y\bar x)\text{ and}\tag{R5}\\
[T_{i,-i}(x),T_{-i,j}(y)]&=T_{ij}(xy)T_{-j,j}(-\lambda^{(\epsilon(j)-\epsilon(-i))/2)}\bar{y}x y)\tag{R6}\end{align*}
hold where $h\neq j,-i$ and $k\neq i,-j$ in \textnormal{(R3)}, $i,k\neq\pm j$ and $i\neq \pm k$ in \textnormal{(R4)} and $i\neq\pm j$ in \textnormal{(R5)} and \textnormal{(R6)}.
\end{lemma}
\begin{proof}
Straightforward calculation.
\end{proof}
\begin{definition}
Let $v\in V$ be isotropic (i.e. $q(v)=0$) such that $v_{-1}=0$. Then we denote the matrix
\begin{align*}
&\left(\begin{array}{cccc|cccc}1&-\bar v_{-2}&\dots&-\bar v_{-n}&-\bar\lambda\bar v_n&\dots&-\bar\lambda\bar v_2&v_1-\bar\lambda\bar v_1\\&1&&&&&&v_2\\&&\ddots&&&&&\vdots\\&&&1&&&&v_n\\\hline&&&&1&&&v_{-n}\\&&&&&\ddots&&\vdots\\&&&&&&1&v_{-2}\\&&&&&&&1\end{array}\right)\\
=&e+ve^t_{-1}- e_1\bar\lambda\tilde v=T_{1,-1}(\bar\lambda|v|+v_1-\bar\lambda\bar v_1)\prod\limits_{i=2}^{-2}T_{i,-1}(v_i)\in EU_{2n}(R,\Lambda)
\end{align*}
by $T_{*,-1}(v)$. Clearly $T_{*,-1}(v)^{-1}=T_{*,-1}(-v)$ (note that $\tilde vv=0$ since $v$ is isotropic) and
\begin{equation}
^{\sigma}T_{*,-1}(v)=e+\sigma v\sigma'_{-1,*}- \sigma_{*1} \bar\lambda\tilde v\sigma^{-1}\overset{L.\ref{35}}{=}e+\sigma v\widetilde{\sigma_{*1}}- \sigma_{*1}\bar\lambda \widetilde{\sigma v}\label{e3}
\end{equation} 
for any $\sigma\in U_{2n}(R,\Lambda)$.
\end{definition}
\begin{definition}\label{38}
Let $i,j\in\Omega$ such that $i\neq\pm j$. Define $P_{ij}:=e+e^{ij}-e^{ji}+\lambda^{(\epsilon(i)-\epsilon(j))/2}e^{-i,-j}-\lambda^{(\epsilon(j)-\epsilon(i))/2}e^{-j,-i}-e^{ii}-e^{jj}-e^{-i,-i}-e^{-j,-j}=T_{ij}(1)T_{ji}(-1)T_{ij}(1)\in EU_{2n}(R,\Lambda)$. It is easy show that $(P_{ij})^{-1}=P_{ji}$.
\end{definition}
\begin{lemma}\label{39}
Let $x\in R$ and $i,j,k\in\Omega$ such that $i\neq \pm j$ and $k\neq \pm i,\pm j$. Further let $y\in\lambda^{-(\epsilon(i)+1)/2}\Lambda$. Then 
\begin{enumerate}[(i)]
\item $^{P_{ki}}T_{ij}(x)=T_{kj}(x)$,
\item $^{P_{kj}}T_{ij}(x)=T_{ik}(x)$ and
\item $^{P_{-k,-i}}T_{i,-i}(y)=T_{k,-k}(\lambda^{(\epsilon(i)-\epsilon(k))/2}y)$.
\end{enumerate}
\end{lemma}
\begin{proof}
Follows from the relations in Lemma \ref{37}.
\end{proof}
\begin{lemma}\label{new}
Let $\sigma\in U_{2n}(R,\Lambda)$ and $i,j\in \Omega$ such that $i\neq \pm j$. Set $\hat\sigma:={}^{P_{ij}}\sigma$. Then
\[|\hat\sigma_{*i}|=
\begin{cases}
|\sigma_{*j}|, \text{ if } \epsilon(i)=\epsilon(j),\\
|\sigma_{*j}|-\bar{\sigma}_{ij}\sigma_{-i,j}+\lambda\overline{\bar{\sigma}_{ij}\sigma_{-i,j}}-\bar{\sigma}_{-j,j}\sigma_{jj}+\lambda\overline{\bar{\sigma}_{-j,j}\sigma_{jj}}, \text{ if } \epsilon(i)=1,\epsilon(j)=-1,\\
|\sigma_{*j}|-\bar{\sigma}_{-i,j}\sigma_{ij}+\lambda\overline{\bar{\sigma}_{-i,j}\sigma_{ij}}-\bar{\sigma}_{jj}\sigma_{-j,j}+\lambda\overline{\bar{\sigma}_{jj}\sigma_{-j,j}}, \text{ if } \epsilon(i)=-1, \epsilon(j)=1.
\end{cases}
\]
\end{lemma}
\begin{proof}
Straightforward computation.
\end{proof}
\subsection{Congruence subgroups}
\begin{definition}
The group consisting of all $\sigma\in U_{2n}(R,\Lambda)$ such that $\sigma\equiv e~mod~I$ and $f(\sigma v,\sigma v)\equiv f(v,v)~mod~\Gamma~\forall v\in V$ is called {\it principal congruence subgroup of level $(I,\Gamma)$} and is denoted by $U_{2n}((R,\Lambda),(I,\Gamma))$. By a theorem of Bak \cite{bak_2}, 4.1.4, cf. \cite{bak-vavilov}, 4.4, it is a normal subgroup of $U_{2n}(R,\Lambda)$.
\end{definition}
\begin{lemma}\label{41}
Let $\sigma\in U_{2n}(R,\Lambda)$. Then $\sigma\in U_{2n}((R,\Lambda),(I,\Gamma))$ if and only if
\begin{enumerate}[(i)]
\item $\sigma\equiv e~mod~I$ and
\item $|\sigma_{*j}|\in\Gamma ~\forall j\in\Omega$.
\end{enumerate}
\end{lemma}
\begin{proof}
\cite{bak-vavilov}, p.174.
\end{proof}
\begin{definition}
The subgroup
\[\{\sigma\in U_{2n}(R,\Lambda)\mid[\sigma,EU_{2n}(R,\Lambda)]\subseteq U_{2n}((R,\Lambda),(I,\Gamma))\}\]
of $U_{2n}(R,\Lambda)$ is called {\it full congruence subgroup of level $(I,\Gamma)$} and is denoted by $CU_{2n}((R,\Lambda),(I,\Gamma))$. Obviously $U_{2n}((R,\Lambda),(I,\Gamma))\subseteq CU_{2n}((R,\Lambda),(I,\Gamma))$. If $EU_{2n}(R,\Lambda)$ is a normal subgroup of $U_{2n}(R,\Lambda)$ (which for example is true if $n\geq 3$ and $R$ is almost commutative, see \cite[Theorem 1.1]{bak-vavilov}), then $CU_{2n}((R,\Lambda),(I,\Gamma))$ is a normal subgroup of $U_{2n}(R,\Lambda)$.
\end{definition}
\begin{theorem}\label{43}
If $n\geq 3$ and $R$ is almost commutative (i.e. module finite over its center), then the equalities
\begin{align*}[CU_{2n}((R,\Lambda),(I,\Gamma)),EU_{2n}(R,\Lambda)]
=[EU_{2n}((R,\Lambda),(I,\Gamma)),EU_{2n}(R,\Lambda)]
=EU_{2n}((R,\Lambda),(I,\Gamma))\end{align*}
hold.
\end{theorem}
\begin{proof}
By \cite[Theorem 1.1]{bak-vavilov}), $EU_{2n}((R,\Lambda),(I,\Gamma))$ is normal in $U_{2n}(R,\Lambda)$ and 
\begin{equation}
[U_{2n}((R,\Lambda),(I,\Gamma)),EU_{2n}(R,\Lambda)]\subseteq EU_{2n}((R,\Lambda),(I,\Gamma))
\end{equation}
(note that in \cite{bak-vavilov} the full congruence subgroup is defined a little differently). By \cite[Lemma 5.2]{bak-vavilov}, 
\begin{equation}
[EU_{2n}((R,\Lambda),(I,\Gamma)),EU_{2n}(R,\Lambda)]=EU_{2n}((R,\Lambda),(I,\Gamma)).
\end{equation}
Hence
\begin{align}
&[CU_{2n}((R,\Lambda),(I,\Gamma)),EU_{2n}(R,\Lambda)]\notag\\
=&[EU_{2n}(R,\Lambda),CU_{2n}((R,\Lambda),(I,\Gamma))]\notag\\\
=&[[EU_{2n}(R,\Lambda),EU_{2n}(R,\Lambda)],CU_{2n}((R,\Lambda),(I,\Gamma))]\notag\\\
\subseteq &EU_{2n}((R,\Lambda),(I,\Gamma)
\end{align}
by the definition of $CU_{2n}((R,\Lambda),(I,\Gamma))$, (4) and the three subgroups lemma. (5) and (6) imply the assertion of the theorem. 
\end{proof}
\section{Sandwich classification for $U_{2n}(R,\Lambda)$}\label{sec8}
In this section $n$ denotes a natural number greater or equal to $3$ and $(R,\Lambda)$ a form ring where $R$ is commutative. 
\begin{definition}
Let $\sigma\in U_{2n}(R,\Lambda)$. Then a matrix of the form $^{\epsilon}\sigma^{\pm 1}$ where $\epsilon\in EU_{2n}(R,\Lambda)$ is called an {\it elementary (unitary) $\sigma$-conjugate}.
\end{definition}
\begin{theorem}\label{mthm3}
Let $\sigma\in U_{2n}(R,\Lambda)$, $k\neq\pm l$ and $i\neq \pm j$. Then 
\begin{enumerate}[(i)]
\item $T_{kl}(\sigma_{ij})$ is a product of $160$ elementary unitary $\sigma$-conjugates,
\item $T_{kl}(\sigma_{i,-i})$ is a product of $320$ elementary unitary $\sigma$-conjugates,
\item $T_{kl}(\sigma_{ii}-\sigma_{jj})$ is a product of $480$ elementary unitary $\sigma$-conjugates,
\item $T_{kl}(\sigma_{ii}-\sigma_{-i,-i})$ is a product of $960$ elementary unitary $\sigma$-conjugates and 
\item $T_{k,-k}(\lambda^{-(\epsilon(k)+1)/2}|\sigma_{*j}|)$ is a product of $1600n+4004$ elementary unitary $\sigma$-conjugates.
\end{enumerate}
\end{theorem}
\begin{proof}(i) In step 1 below we show that $T_{kl}(x\bar\sigma_{23}\sigma_{2,-1})$ where $x\in R$ is a product of $16$ elementary $\sigma$-conjugates. In step 2 we show that $T_{kl}(x\bar\sigma_{23}\sigma_{21})$ where $x\in R$ is a product of $16$ elementary $\sigma$-conjugates. In step 3 we show that $T_{kl}(x\bar\sigma_{23}\sigma_{22})$ is a product of $32$ elementary $\sigma$-conjugates. In step 4 we use steps 1-3 in order to prove (i).\\
\\
\underline{step 1} Set $\tau:=T_{21}(\bar\sigma_{23}\sigma_{23})T_{31}(-\bar\sigma_{23}\sigma_{22})T_{3,-2}(\bar\sigma_{23}\sigma_{2,-1})T_{3,-3}(-\bar\sigma_{22}\sigma_{2,-1}+\bar\lambda\bar\sigma_{2,-1}\sigma_{22})$. One checks easily that the second row of $\sigma\tau^{-1}$ equals the second row of $\sigma$ and hence the second row of 
$\xi:={}^{\sigma}\tau^{-1}$ is trivial. By Lemma \ref{33} the second last column of $\xi$ also is trivial. Set  
\begin{align*}
\zeta:={}^{\tau^{-1}}[T_{-1,2}(1),[\tau,\sigma]]={}^{\tau^{-1}}[T_{-1,2}(1),\tau\xi]\overset{L.\ref{pre}}{=}[\tau^{-1},T_{-1,2}(1)][T_{-1,2}(1),\xi].
\end{align*}
One checks easily that $[\tau^{-1},T_{-1,2}(1)]=T_{31}(\lambda\bar\sigma_{23}\sigma_{2,-1})T_{-1,1}(z)$ for some $z\in \Lambda$ and $[T_{-1,2}(1),\xi]=\prod\limits_{i\neq 2 }T_{i2}(x_i)$ for some $x_i\in R~(i\neq 2)$. Hence $\zeta=T_{31}(\lambda\bar\sigma_{23}\sigma_{2,-1})T_{-1,1}(z)\prod\limits_{i\neq 2 }T_{i2}(x_i)$. It follows that $[T_{-1,3}(-x\bar\lambda),[T_{12}(1),\zeta]]=T_{-1,2}(x\bar\sigma_{23}\sigma_{2,-1})$ for any $x\in R$. Hence we have shown
\[[T_{-1,3}(-x\bar\lambda),[T_{12}(1),{}^{\tau^{-1}}[T_{-1,2}(1),[\tau,\sigma]]]]=T_{-1,2}(x\bar\sigma_{23}\sigma_{2,-1}).\]
This implies that $T_{-1,2}(x\bar\sigma_{23}\sigma_{2,-1})$ is a product of $16$ elementary $\sigma$-conjugates. It follows from Lemma \ref{39} that $T_{kl}(x\bar\sigma_{23}\sigma_{2,-1})$ is a product of $16$ elementary $\sigma$-conjugates.\\
\\
\underline{step 2} Set $\tau:=T_{1,-2}(\bar\sigma_{23}\sigma_{23})T_{3,-2}(-\bar\sigma_{23}\sigma_{21})T_{3,-1}(\bar\lambda\bar\sigma_{23}\sigma_{22})T_{3,-3}(\bar\sigma_{22}\sigma_{21}-\bar\lambda\bar\sigma_{21}\sigma_{22})$. One checks easily that the second row of $\sigma\tau^{-1}$ equals the second row of $\sigma$ and hence the second row of $\xi:={}^{\sigma}\tau^{-1}$ is trivial. By Lemma \ref{33} the second last column of $\xi$ also is trivial. Set  
\begin{align*}
\zeta:={}^{\tau^{-1}}[T_{-2,-1}(1),[\tau,\sigma]]={}^{\tau^{-1}}[T_{-2,-1}(1),\tau\xi]\overset{L.\ref{pre}}{=}[\tau^{-1},T_{-2,-1}(1)][T_{-2,-1}(1),\xi].
\end{align*}
One checks easily that $[\tau^{-1},T_{-2,-1}(1)]=T_{3,-1}(\bar\sigma_{23}\sigma_{21})T_{1,-1}(z)$ for some $z\in \bar\Lambda$ and $[T_{-2,-1}(1),\xi]=\prod\limits_{i\neq 2 }T_{i2}(x_i)$ for some $x_i\in R~(i\neq 2)$. Hence $\zeta=T_{3,-1}(\bar\sigma_{23}\sigma_{21})T_{1,-1}(z)\prod\limits_{i\neq 2 }T_{i2}(x_i)$. It follows that $[T_{-1,3}(-x),$ $[T_{-2,3}(1),\zeta]]=T_{-2,3}(x\bar\sigma_{23}\sigma_{21})$ for any $x\in R$. Hence we have shown
\[[T_{-1,3}(-x),[T_{-2,3}(1),{}^{\tau^{-1}}[T_{-2,-1}(1),[\tau,\sigma]]]]=T_{-2,3}(x\bar\sigma_{23}\sigma_{21}).\]
This implies that $T_{-2,3}(x\bar\sigma_{23}\sigma_{21})$ is a product of $16$ elementary $\sigma$-conjugates. It follows from Lemma \ref{39} that $T_{kl}(x\bar\sigma_{23}\sigma_{21})$ is a product of $16$ elementary $\sigma$-conjugates.\\
\\
\underline{step 3} Set $\tau:=T_{21}(-\bar\sigma_{22}\sigma_{23})T_{31}(\bar\sigma_{22}\sigma_{22})T_{2,-3}(\bar\sigma_{22}\sigma_{2,-1})T_{2,-2}(-\bar\sigma_{23}\sigma_{2,-1}+\bar\lambda\bar\sigma_{2,-1}\sigma_{23})$. One checks easily that the second row of $\sigma\tau^{-1}$ equals the second row of $\sigma$ and hence the second row of 
$\xi:={}^{\sigma}\tau^{-1}$ is trivial. By Lemma \ref{33} the second last column of $\xi$ also is trivial. Set  
\begin{align*}
\zeta:={}^{\tau^{-1}}[T_{32}(1),[\tau,\sigma]]={}^{\tau^{-1}}[T_{32}(1),\tau\xi]\overset{L.\ref{pre}}{=}[\tau^{-1},T_{32}(1)][T_{32}(1),\xi].
\end{align*}
One checks easily that $\psi:=[\tau^{-1},T_{32}(1)]=T_{31}(-\bar\sigma_{22}\sigma_{23})T_{3,-3}(y)T_{3,-2}(z)$ for some $y\in\bar \Lambda$ and $z\in R$ and
$\theta:=[T_{32}(1),\xi]=\prod\limits_{i\neq 2 }T_{i2}(x_i)$ for some $x_i\in R~(i\neq 2)$. 
Set 
\[\chi:={}^{\psi^{-1}}[T_{12}(1),\zeta]={}^{\psi^{-1}}[T_{12}(1),\psi\theta]\overset{L.\ref{pre}}{=}[\psi^{-1},T_{12}(1)][T_{12}(1),\theta].\]
One checks easily that $[\psi^{-1},T_{12}(1)]=T_{32}(\bar\sigma_{22}\sigma_{23})T_{3,-3}(a)T_{3,-1}(b)$ for some $a\in\bar \Lambda$ and $b\in R$ and $[T_{12}(1),\theta]=T_{-2,2}(d)$ for some $d\in \Lambda$. Hence $\chi=T_{32}(\bar\sigma_{22}\sigma_{23})T_{3,-3}(a)T_{3,-1}(b)T_{-2,2}(d)$. It follows that $[T_{-2,3}(\bar x),[T_{2,-1}(1),\chi]]=T_{-2,-1}(-\bar x\bar\sigma_{22}\sigma_{23})\overset{(R1)}{=}T_{12}(x\bar\sigma_{23}\sigma_{22})$ for any $x\in R$. Hence we have shown
\[[T_{-2,3}(\bar x),[T_{2,-1}(1),{}^{\psi^{-1}}[T_{12}(1),{}^{\tau^{-1}}[T_{32}(1),[\tau,\sigma]]]]]=T_{12}(x\bar\sigma_{23}\sigma_{22}).\]
This implies that $T_{12}(x\bar\sigma_{23}\sigma_{22})$ is a product of $32$ elementary $\sigma$-conjugates. It follows from Lemma \ref{39} that $T_{kl}(x\bar\sigma_{23}\sigma_{22})$ is a product of $32$ elementary $\sigma$-conjugates.\\
\\
\underline{step 4} Set $I:=I(\{\overline{\bar\sigma_{23}\sigma_{2,-1}},\bar\sigma_{23}\sigma_{21}\})$, $J:=I(\{\overline{\bar\sigma_{23}\sigma_{2,-1}},\bar\sigma_{23}\sigma_{21},\bar\sigma_{23}\sigma_{22}\})$ and
\[\tau:=[\sigma^{-1},T_{12}(-\bar\sigma_{23})]=(e-\sigma'_{*1}\bar\sigma_{23}\sigma_{2*}+\sigma_{*,-2}'\sigma_{23}\sigma_{-1,*})T_{12}(\bar\sigma_{23}).\]
One checks easily that $\tau_{11}\equiv 1 ~mod~I$ and $\tau_{12}\equiv \bar\sigma_{23}~mod~J$. Set 
$\zeta:={}^{P_{13}P_{21}}\tau$. Then $\zeta_{22}=\tau_{11}$ and $\zeta_{23}=\tau_{12}$ and hence $\bar\zeta_{23}\zeta_{22}\equiv \sigma_{23}~mod~I+\bar J$. Applying step 3 above to $\zeta$, we get that $T_{kl}(\bar\zeta_{23}\zeta_{22})$ is a product of $32$ elementary $\zeta$-conjugates. Since any elementary $\zeta$-conjugate is a product of $2$ elementary $\sigma$-conjugates, it follows that $T_{kl}(\bar\zeta_{23}\zeta_{22})$ is a product of $64$ elementary $\sigma$-conjugates. Thus, by steps 1-3, $T_{kl}(\sigma_{23})$ is a product of $64+16+16+16+16+32=160$ elementary $\sigma$-conjugates. Since one can bring $\sigma_{ij}$ to position $(3,2)$ by conjugating monomial matrices in $EU_{2n}(R,\Lambda)$, the assertion of (i) follows.\\
\\
(ii)-(iv) See the proof of Theorem \ref{mthm2}.\\
\\
(v) Set m:=160. In step 1 we show that $T_{k,-k}(\lambda^{-(\epsilon(k)+1)/2}\bar x\bar\sigma_{11}|\sigma_{*1}|\sigma_{11}x)$ where $x\in R$ is a product of $(2n+17)m+4$ elementary $\sigma$-conjugates. In step 2 we use step 1 in order to prove (v).\\
\\
\underline{step 1} Set $v':=\begin{pmatrix}0&\dots&0&\sigma'_{-1,-1}&-\sigma'_{-1,-2}\end{pmatrix}^t=\begin{pmatrix}0&\dots&0&\bar\sigma_{11}&-\bar\sigma_{21}\end{pmatrix}^t\in V$ and $v:=\sigma^{-1}v'\in V$. Then clearly $v_{-1}=0$. Further $q(v)=q(\sigma^{-1}v')=q(v')=0$ and hence $v$ is isotropic. Set
\[
\xi:={}^{\sigma}T_{*,-1}(-v)\overset{(\ref{e3})}{=}e-\sigma v\widetilde{\sigma_{*1}}+ \sigma_{*1} \bar\lambda\widetilde{\sigma v}=e-v'\widetilde{\sigma_{*1}}+ \sigma_{*1}\bar\lambda \widetilde{v'}.
\]
Then
\begin{align*}
\xi=
\arraycolsep=8pt\def\arraystretch{1.5}\left(\begin{array}{cccccc|cccccc}
*&\sigma_{11}\sigma_{11}&&&&&&&&&&\\
*&1+\sigma_{21}\sigma_{11}&&&&&&&&&&\\
*&\sigma_{31}\sigma_{11}&1&&&&&&&&&\\
*&\sigma_{41}\sigma_{11}&&1&&&&&&&\\
\vdots&\vdots&&&\ddots&&&&&&&\\
*&\sigma_{n1}\sigma_{11}&&&&1&&&&&&\\
\hline*&\sigma_{-n,1}\sigma_{11}&&&&&1&&&&\\
\vdots&\vdots&&&&&&\ddots&&&&\\
*&\sigma_{-4,1}\sigma_{11}&&&&&&&1&&&\\
-\sigma_{-3,1}\sigma_{21}&\sigma_{-3,1}\sigma_{11}&&&&&&&&1&&\\
*&\alpha&*&*&\dots&*&*&\dots&*&-\bar\sigma_{11}\bar\sigma_{31}&*&*\\
*&\beta&*&*&\dots&*&*&\dots&*&\bar\sigma_{21}\bar\sigma_{31}&*&*
\end{array}\right)
\end{align*}
where $\alpha=\sigma_{-2,1}\sigma_{11}-\lambda\bar\sigma_{11}\bar\sigma_{-2,1}$ and $\beta=\sigma_{-1,1}\sigma_{11}+\lambda\bar\sigma_{21}\bar\sigma_{-2,1}$. Set \[\tau:=T_{-3,1}(\sigma_{-3,1}\sigma_{21})T_{-3,2}(-\sigma_{-3,1}\sigma_{11}).\] 
It follows from (i) that $\tau$ is a product of $2m$ elementary $\sigma$-conjugates. Clearly
\begin{align*}
\xi\tau=
\arraycolsep=8pt\def\arraystretch{1.5}\left(\begin{array}{cccccc|cccccc}
*&\sigma_{11}\sigma_{11}&&&&&&&&&&\\
*&1+\sigma_{21}\sigma_{11}&&&&&&&&&&\\
*&\sigma_{31}\sigma_{11}&1&&&&&&&&&\\
*&\sigma_{41}\sigma_{11}&&1&&&&&&&\\
\vdots&\vdots&&&\ddots&&&&&&&\\
*&\sigma_{n1}\sigma_{11}&&&&1&&&&&&\\
\hline*&\sigma_{-n,1}\sigma_{11}&&&&&1&&&&\\
\vdots&\vdots&&&&&&\ddots&&&&\\
*&\sigma_{-4,1}\sigma_{11}&&&&&&&1&&&\\
0&0&&&&&&&&1&&\\
*&\gamma&0&*&\dots&*&*&\dots&*&*&*&*\\
*&\delta&0&*&\dots&*&*&\dots&*&*&*&*
\end{array}\right)
\end{align*}
where $\gamma=\alpha+\bar\sigma_{11}\bar\sigma_{31}\sigma_{-3,1}\sigma_{11}$ and $\delta=\beta-\bar\sigma_{21}\bar\sigma_{31}\sigma_{-3,1}\sigma_{11}$. Let $x\in R$ and set  
\begin{align*}
\zeta:={}&^{T_{*,-1}(-v)}[T_{2,-3}(-x),[T_{*,-1}(v),\sigma]\tau]\\
={}&^{T_{*,-1}(-v)}[T_{2,-3}(-x),T_{*,-1}(v)\xi\tau]
\\\overset{L.\ref{pre}}{=}&[T_{*,-1}(-v),T_{2,-3}(-x)][T_{2,-3}(-x),\xi\tau].
\end{align*}
Clearly $\zeta$ is a product of $4m+4$ elementary $\sigma$-conjugates. One checks easily that 
\begin{align*}
&[T_{*,-1}(-v),T_{2,-3}(-x)]\\
=&T_{1,-1}(\bar\lambda(
-\bar xv_{-2}\bar v_{-3}+\lambda\overline{\bar xv_{-2}\bar v_{-3}}))T_{1,-2}(\bar\lambda\bar x \bar v_{-3})T_{1,-3}(-x\bar v_{-2})\\
=&T_{1,-1}(\bar\lambda(a+\lambda\bar a))T_{1,-2}(b+c)T_{1,-3}(-x(\sigma_{11}\sigma_{22}-\sigma_{12}\sigma_{21}))
\end{align*}
for some $a,b\in I(\sigma_{21}), c\in I(\sigma_{23})$. Further
\[[T_{2,-3}(-x),\xi\tau]=(\prod\limits_{\substack{p=1,\\p\neq 3}}^{-4}T_{p,-3}(x\sigma_{p1}\sigma_{11}))T_{-2,-3}(x\gamma)T_{-1,-3}(x\delta)T_{3,-3}(y)\]
where $y=\bar\lambda(\bar x\bar\sigma_{11}|\sigma_{*1}|\sigma_{11}x+d-\lambda\bar d+e-\lambda\bar e)$ for some $d\in I(\sigma_{31}),e\in I(\sigma_{-2,1})$. Hence
\begin{align*}
\zeta=&T_{1,-1}(\bar\lambda(a+\lambda\bar a))T_{1,-2}(b+c)T_{1,-3}(-x(\sigma_{11}(\sigma_{22}-\sigma_{11})-\sigma_{12}\sigma_{21}))\cdot\\
&\cdot (\prod\limits_{\substack{p=2,\\p\neq 3}}^{-4}T_{p,-3}(x\sigma_{p1}\sigma_{11}))T_{-2,-3}(x\gamma)T_{-1,-3}(x\delta)T_{3,-3}(y).
\end{align*}
It follows from (i), (ii) and (iii) and relation (R5) in Lemma \ref{37} that $T_{3,-3}(y)$ is a product of $4m+4+2m+2m+4m+(2n-5)m+3m+3m=(2n+13)m+4$ elementary $\sigma$-conjugates. By (i) and relation (R5) in Lemma \ref{37}, $T_{3,-3}(-\bar\lambda(d-\lambda\overline{d}))$ and $T_{3,-3}(-\bar\lambda(e-\lambda\overline{e})))$ each are a product of $2m$ elementary $\sigma$-conjugates. Hence $T_{3,-3}(\bar\lambda(\bar x\bar\sigma_{11}|\sigma_{*1}|\sigma_{11}x))=T_{3,-3}(y)T_{3,-3}(-\bar\lambda(d-\lambda\overline{d}))T_{3,-3}(-\bar\lambda(e-\lambda\overline{e})))$ is a product of $(2n+17)m+4$ elementary $\sigma$-conjugates. It follows from Lemma \ref{39} that $T_{k,-k}(\lambda^{-(\epsilon(k)+1)/2}\bar x\bar\sigma_{11}|\sigma_{*1}|\sigma_{11}x)$ is a product of $(2n+17)m+4$ elementary $\sigma$-conjugates.\\
\\
\underline{step 2} Clearly
\begin{align*}
&T_{k,-k}(\lambda^{-(\epsilon(k)+1)/2}|\sigma_{*1}|)\\
=&T_{k,-k}(\lambda^{-(\epsilon(k)+1)/2}\overline{\sum\limits_{q\in \Omega}\sigma'_{1q}\sigma_{q1}}|\sigma_{*1}|\sum\limits_{r\in \Omega}\sigma'_{1r}\sigma_{r1})\\
=&T_{k,-k}(\lambda^{-(\epsilon(k)+1)/2}\sum\limits_{q,r\in \Omega}\bar\sigma'_{1q}\bar\sigma_{q1}|\sigma_{*1}|\sigma_{r1}\sigma'_{1r})\\
=&T_{k,-k}(\lambda^{-(\epsilon(k)+1)/2}\sum\limits_{q}\bar\sigma'_{1q}\bar\sigma_{q1}|\sigma_{*1}|\sigma_{q1}\sigma'_{1q})T_{k,-k}(\lambda^{-(\epsilon(k)+1)/2}\sum\limits_{q<r}\bar\sigma'_{1q}\bar\sigma_{q1}|\sigma_{*1}|\sigma_{r1}\sigma'_{1r}+\bar\sigma'_{1r}\bar\sigma_{r1}|\sigma_{*1}|\sigma_{q1}\sigma'_{1q})\\
=&\underbrace{T_{k,-k}(\lambda^{-(\epsilon(k)+1)/2}\sum\limits_{q}\bar\sigma'_{1q}\bar\sigma_{q1}|\sigma_{*1}|\sigma_{q1}\sigma'_{1q})}_{A:=}\underbrace{T_{k,-k}(\lambda^{-(\epsilon(k)+1)/2}\sum\limits_{q<r}\bar\sigma'_{1q}\bar\sigma_{q1}|\sigma_{*1}|\sigma_{r1}\sigma'_{1r}-\lambda\overline{\bar\sigma'_{1q}\bar\sigma_{q1}|\sigma_{*1}|\sigma_{r1}\sigma'_{1r}})}_{B:=}
\end{align*}
since $|\sigma_{*1}|\in \Lambda\subseteq \Lambda_{max}$. By step 1, $T_{k,-k}(\lambda^{-(\epsilon(k)+1)/2}\bar\sigma'_{11}\bar\sigma_{11}|\sigma_{*1}|\sigma_{11}\sigma'_{11})$ is a product of $(2n+17)m+4$ elementary $\sigma$-conjugates. By (i), (ii) and relation (R6) in Lemma \ref{37}, $T_{k,-k}(\lambda^{-(\epsilon(k)+1)/2}\bar\sigma'_{1q}\bar\sigma_{q1}|\sigma_{*1}|\sigma_{q1}\sigma'_{1q})$ is a product of $3m$ elementary $\sigma$-conjugates if $q\neq \pm 1$ resp. a product of $6m$ elementary $\sigma$-conjugates if $q=-1$. Hence $A$ is a product of $(2n+17)m+4+(2n-2)\cdot 3m+6m=(8n+17)m+4$ elementary $\sigma$-conjugates. On the other hand $B=T_{k,-k}(\lambda^{-(\epsilon(k)+1)/2}(x-\lambda\bar x))$ where $x\in I(|\sigma_{*1}|)$. Since $|\sigma_{*1}|=\sum\limits_{i\in \Omega_+}\bar\sigma_{i1}\sigma_{-i,1}$, it follows from (i), (ii) and relation (R5) in Lemma \ref{37} that $B$ is a product of $4m+(n-1)\cdot 2m=(2n+2)m$ elementary $\sigma$-conjugates. Hence $T_{k,-k}(\lambda^{-(\epsilon(k)+1)/2}|\sigma_{*1}|)$ is a product of $(10n+19)m+4=1600n+3044$ elementary $\sigma$-conjugates. The assertion of (v) follows now from Lemma \ref{new}.
\end{proof}

As a corollary we get the Sandwich Classification Theorem for $U_{2n}(R,\Lambda)$.
\begin{corollary}
Let $H$ be a subgroup of $U_{2n}(R,\Lambda)$. Then $H$ is normalized by $EU_{2n}(R,\Lambda)$ if and only if 
\begin{equation}
EU_{2n}((R,\Lambda),(I,\Gamma))\subseteq H\subseteq CU_{2n}((R,\Lambda),(I,\Gamma))
\end{equation}
for some form ideal $(I,\Gamma)$ of $(R,\Lambda)$.
\end{corollary}
\begin{proof}
First suppose that $H$ is is normalized by $EU_{2n}(R,\Lambda)$. Let $(I,\Gamma)$ be the form ideal of $(R,\Lambda)$ defined by $I:=\{x\in R\mid T_{12}(x)\in H\}$ and $\Gamma:=\{y\in \Lambda\mid T_{-1,1}(y)\in H\}$. Then clearly $EU_{2n}((R,\Lambda),(I,\Gamma))\subseteq H$. It remains to show that $H\subseteq CU_{2n}((R,\Lambda),(I,\Gamma))$, i.e. that if $\sigma\in H$ and $\epsilon\in EU_{2n}(R,\Lambda)$, then $[\sigma,\epsilon]\in U_{2n}((R,\Lambda),(I,\Gamma))$. By Lemma \ref{41} it suffices to show that if $\sigma\in H$ and $\epsilon\in EU_{2n}(R,\Lambda)$, then $[\sigma,\epsilon]\equiv e~mod~I$ and $|[\sigma,\epsilon]_{*j}|\in \Gamma$ for any $j\in\Omega$. But that follows from the previous theorem (applying the theorem to $\sigma$ we get that $\sigma\equiv diag(x,\dots,x)~mod~I$ for some $x\in R$ and hence $[\sigma,\epsilon]\equiv e~mod~I$; applying it to $[\sigma,\epsilon]$ we get that $|[\sigma,\epsilon]_{*j}|\in \Gamma$ for any $j\in\Omega$). Suppose now that (7) holds for some form ideal $(I,\Gamma)$. Then it follows from the standard commutator formulas in Theorem \ref{43} that $H$ is normalized by $EU_{2n}(R,\Lambda)$.
\end{proof}

\end{document}